\documentclass[12pt]{amsart}
\usepackage{latexsym}
\usepackage{amscd,amssymb,times,fullpage}
\newtheorem{thm}{Theorem}[section]
\newtheorem{prop}[thm]{Proposition}
\newtheorem{lem}[thm]{Lemma}
\newtheorem{cor}[thm]{Corollary}

\theoremstyle{remark}
\newtheorem{rem}[thm]{Remark}
\newtheorem{ex}[thm]{Example}

\theoremstyle{definition}
\newtheorem{defn}[thm]{Definition}

\numberwithin{equation}{section}

\newcommand{\bR}{\mathbb R}
\newcommand{\bC}{\mathbb C}
\newcommand{\bP}{\mathbb P}

\newcommand{\blowup}{\overline{\bC P^2}}

\newcommand{\C}{\mathbb{ C}}
\newcommand{\del}{\partial}
\newcommand{\Hom}{\operatorname{Hom}}

\newcommand{\OO}{\mathcal{ O}}

\newcommand{\Z}{\mathbb{ Z}}

\title{Fibrations and fundamental groups of K\"ahler--Weyl manifolds}
\author{G.~Kokarev}
\address{School of Mathematics, The University of Edinburgh, King's
  Buildings, Mayfield Road, Edinburgh EH9 3JZ, UK}
\email{G.Kokarev@ed.ac.uk}
\address[Current address]{Mathematisches Institut, Ludwig-Maximilians-Universit\"at M\"unchen,
Theresienstrasse 39, 80333 M\"unchen, Germany}
\email{Gerasim.Kokarev@mathematik.uni-muenchen.de}

\author{D.~Kotschick}
\address{Mathematisches Institut, Ludwig-Maximilians-Universit\"at M\"unchen,
Theresienstrasse 39, 80333 M\"unchen, Germany}
\email{dieter@member.ams.org}
\date{November 10, 2008, revised July 8, 2009; \copyright{\ G.~Kokarev and D.~Kotschick 2008}}

\begin{document}

\begin{abstract}
We extend the Siu--Beauville theorem to a certain class of compact
K\"ahler--Weyl manifolds, proving that they fiber holomorphically over
hyperbolic Riemannian surfaces whenever they satisfy the necessary
topological hypotheses. As applications we obtain restrictions on the
fundamental groups of such K\"ahler--Weyl manifolds, and show that in
certain cases they are in fact K\"ahler.
\end{abstract}

\maketitle

\section{Introduction}\label{s:intro}

There are many results concerning the topology of complex algebraic
varieties and of compact K\"ahler manifolds that are proved using
analytic methods like Hodge theory and harmonic maps. These methods
strongly depend on differential-geometric features of K\"ahler
manifolds that ultimately derive from the K\"ahler identities. Although 
the methods do not immediately extend, sometimes such results do
generalize to non-K\"ahler compact complex surfaces using a
case-by-case analysis appealing to the Enriques--Kodaira
classification. It turns out that K\"ahler manifolds and complex
surfaces are both special cases of complex manifolds supporting
K\"ahler--Weyl structures, and some results that were previously
known with disparate proofs in those two special cases can actually be
proved uniformly for a certain class of K\"ahler--Weyl manifolds.

The notion of a K\"ahler--Weyl structure arises naturally in conformal
geometry, and goes back several decades to papers of Vaisman. We shall
give a brief account of the basic definitions in Section~\ref{s:pre}
below, and refer the reader to~\cite{CP,DO} and the references cited
there for further details. The upshot is that the K\"ahler--Weyl
condition is vacuous in complex dimensions one and two, just like the
K\"ahler condition is vacuous in dimension one, and is equivalent to
the locally conformally K\"ahler condition in higher dimensions.

Recently the first author extended some of the harmonic map techniques
from K\"ahler geometry to the more general setting of K\"ahler--Weyl
geometry, see~\cite{Kok}. This extension involves the study of
pseudo-harmonic or Weyl harmonic maps from K\"ahler--Weyl domains
endowed with so-called pluricanonical metrics. We recall the notion of a
pluricanonical metric in Section~\ref{s:pre}; it covers a large class
of examples and makes the Bochner technique applicable in the setting of
Weyl harmonic maps.

In this paper we use the techniques and results from~\cite{Kok} to
prove and apply a generalization of the following theorem due to Siu and Beauville:
\begin{thm}\label{t:BSorig}
For a compact K\"ahler manifold $M$ the following two statements are equivalent:
\begin{itemize}
\item[I.] $M$ admits a surjective holomorphic map with connected fibers to a compact Riemann surface of genus $\geq 2$, and
\item[II.] $\pi_1(M)$ admits a surjective homomorphism to the fundamental group of a compact Riemann surface of genus $\geq 2$.
\end{itemize}
\end{thm}
Obviously the first statement implies the second. The converse was first proved by Siu~\cite{Siu} using harmonic maps. 
It was later rediscovered by Beauville, whose proof is a sophisticated application of Hodge theory; see the appendix 
to~\cite{Cat}. It is not hard to see, and is explained in~\cite[Chapter 2]{ABCKT}, that in the situation of Theorem~\ref{t:BSorig} 
the two statements are also equivalent to the following:
\begin{itemize}
\item[III.] {\it $\pi_1(M)$ admits a surjective homomorphism to a non-abelian free group.}
\end{itemize}
A simple proof of the Siu--Beauville theorem is due to Catanese~\cite{Cat}. Firstly, one notes that II and III each imply the following:
\begin{itemize}
\item[IV.] {\it $H^1(M;\bR)$ admits an isotropic subspace of dimension $\geq 2$.}
\end{itemize}
Here a subspace $U\subset H^1(M;\bR)$ is said to be isotropic if the cup product map $\Lambda^2 U\longrightarrow H^2(M;\bR)$
vanishes identically. 
Secondly, from an isotropic subspace as in IV, a straightforward application of Hodge theory produces two holomorphic one-forms 
with trivial wedge product, to which the Castelnuovo-de~Franchis lemma can be applied, cf.~\cite{Cat,ABCKT}. Thus, on a compact 
K\"ahler manifold conditions I, II, III and IV are all equivalent.

For compact complex surfaces it is known that the Siu--Beauville
theorem holds, by appealing to the Kodaira classification in the
non-K\"ahler case, see~\cite[Chapter 2]{ABCKT}. It is also known that
statement IV is strictly weaker than I and II in this case; the
Kodaira--Thurston manifold is a compact complex surface for which
IV holds although the surface does not fiber over a curve of genus $\geq 2$.

The main results of this paper are Theorems~\ref{t:BS} and~\ref{t:main}, 
proved in Section~\ref{s:BS}. Theorem~\ref{t:BS} says that conditions
I, II and III are equivalent for compact pluri-K\"ahler--Weyl manifolds.
This unifies the known results for K\"ahler manifolds and complex
surfaces, and generalizes them to locally conformally K\"ahler manifolds with
pluricanonical metrics. Our argument applies the strategy of
Siu~\cite{Siu} using Weyl harmonic instead of the usual harmonic maps,
and relies on some of the results from~\cite{Kok}. As an immediate
corollary of the proof we will see that non-abelian free groups can
not be the fundamental groups of compact pluri-K\"ahler--Weyl
manifolds. For K\"ahler manifolds this is of course an easy
application of Hodge theory together with covering arguments. For
compact complex surfaces the corollary was also known as a consequence
of classification results, see~\cite[Chapter 2]{ABCKT}. For
pluri-K\"ahler-Weyl manifolds of higher dimension, the corollary is
new. In Theorem~\ref{t:main} we prove that statement I holds more
generally, assuming only that $\pi_1(M)$ admits a representation with 
non-cyclic image in the fundamental group of some hyperbolic manifold,
which does not have to be a surface.

In Section~\ref{s:twistor} we apply the main results to twistor spaces
of (half-)conformally flat manifolds. The conclusion is that if the
fundamental group is large, then a twistor space can not be
pluri-K\"ahler--Weyl. This generalizes various results to the effect
that K\"ahlerian twistor spaces are simply connected, cf.~\cite{BMN,Cam,Hitchin,Slup}.

In Section~\ref{s:appl} we give further applications of
Theorem~\ref{t:BS}. We prove that pluri-K\"ahler--Weyl manifolds with
certain fundamental groups have to be K\"ahler, and give 
obstructions to the existence of pluri-K\"ahler--Weyl structures on
some complex manifolds.

\subsection*{Acknowledgements:}
The first author was supported by an EPSRC fellowship at the University
of Edinburgh. Parts of this work were done while the second author was taking part in the activity on ``Extremal
K\"ahler metrics and K\"ahler--Ricci flow'' at the Centro Ennio De Giorgi in Pisa. He is grateful
to S.~Salamon for the invitation, and for a helpful suggestion in connection with this work.
The paper was completed while the second author enjoyed the support of The Bell Companies 
Fellowship at the Institute for Advanced Study in Princeton.

%\newpage

\section{Preliminaries and background}\label{s:pre}

\subsection{K\"ahler--Weyl geometry}\label{ss:KW}

Let $(M,c)$ be a smooth conformal manifold. A Weyl structure or Weyl
connection on $(M,c)$ is a torsion-free connection $\nabla^W$ which
preserves the conformal structure $c$; this means that for any
$g\in c$ there exists a $1$-form $\theta$ (sometimes called a Higgs
field) such that 
$$
\nabla^W g = \theta\otimes g \ .
$$
A standard calculation shows that this condition is equivalent to 
\begin{equation}\label{eq:weyl}
\nabla^W_X Y = \nabla_X Y - \frac{1}{2}(\theta(X)Y+\theta(Y)X-g(X,Y)\theta^{\sharp}) \ ,
\end{equation}
where $\nabla$ is the Levi-Civita connection of $g$, and
$\theta^{\sharp}$ is the vector field $g$-dual to $\theta$.
Clearly the space of Weyl connections is a non-empty affine space
whose vector space of translations is formed by $1$-forms. 

The property of being Hermitian with respect to a fixed almost complex
structure $J$ is conformally invariant, and thus the following
definition makes sense:
\begin{defn} (\cite{CP})
A K\"ahler--Weyl structure on $M$ is a triple $(J,c,\nabla^W)$,
where $J$ is an almost complex structure, $c$ is a conformal structure
which is Hermitian with respect to $J$, and $\nabla^W$ is a Weyl connection
which preserves $J$, i.~e. $\nabla^W J=0$.
\end{defn}
Since the Weyl connection is torsion-free by definition, the condition
$\nabla^W J=0$ implies that $J$ is integrable. Thus, there is no loss
of generality in assuming from the outset that $(M,J)$ is a complex
manifold. In the special case when the Weyl connection coincides with
the Levi-Civita connection of $g\in c$, the above definition reduces
to one of the standard definitions of a K\"ahler structure. In complex
dimension one every conformal structure is part of a K\"ahler--Weyl
structure which is in fact K\"ahler.

For a given metric $g\in c$ the fundamental two-form $\omega$ of the
Hermitian structure $(M,J,g)$ is defined as usual by 
$$
\omega (X,Y) = g(X,JY) \ .
$$
For a K\"ahler--Weyl structure the defining conditions imply
\begin{equation}\label{eq:omegaint}
d\omega = \omega\wedge\theta \ ,
\end{equation}
which means that $\theta$ is the Lee form of the Hermitian structure $(M,J,g)$.
Moreover, \eqref{eq:omegaint} implies
$$
\omega\wedge d\theta = 0 \ .
$$
In complex dimensions $\geq 3$ the multiplication with the fundamental
$2$-form is injective on $2$-forms, and we conclude that the Lee form
$\theta$ is closed. Therefore in these dimensions any metric $g\in c$
is locally conformally K\"ahler. In more detail, the form $\theta$ is
locally exact by the Poincar\'e lemma, $\theta=df$. The locally
defined metric $e^{-f} g$ is preserved by the Weyl connection
$\nabla^W$ and, hence, is K\"ahler. Conversely, given a locally
conformally K\"ahler metric, the Levi-Civita connections of the
(essentially unique) locally defined K\"ahler metrics fit together to
form a global Weyl connection.

In complex dimension $= 2$, the identity~\eqref{eq:omegaint} is true
(for some $\theta$, uniquely determined by $\omega$) for any Hermitian metric, because the
multiplication with the fundamental two-form is an isomorphism on
$1$-forms. The form $\theta$ defines a Weyl connection by
formula~\eqref{eq:weyl}, which actually preserves the complex
structure.

We summarize this discussion in the following:
\begin{prop}
In complex dimensions $\leq 2$ every complex manifold admits a
K\"ahler--Weyl structure. In complex dimensions $\geq 3$ a complex
manifold admits a K\"ahler--Weyl structure if and only if it admits a
locally conformally K\"ahler structure.
\end{prop}

Let $(M,c)$ be a conformal manifold equipped with a Weyl connection.
Recall that by the theorem of Gauduchon~\cite{Gau} there exists a
canonical metric $g\in c$, unique up to homothety, such that the
corresponding Lee form is co-closed, $tr_g(\nabla\theta)=0$. For
example, for a globally conformally K\"ahler manifold the canonical
metric is K\"ahler and its Lee form vanishes identically. Below we shall
consider K\"ahler--Weyl manifolds whose canonical metrics satisfy
an extra hypothesis introduced in~\cite{Kok}:
\begin{defn}(\cite{Kok})\label{d:pluri}
A metric $g\in c$ on a K\"ahler--Weyl manifold is called
pluricanonical if the $(1,1)$-part of the covariant derivative of the
Lee from vanishes, i.~e.~$(\nabla\theta)^{1,1}=0$.
\end{defn}
Note that all {\it Vaisman manifolds} (also called {\it generalised
  Hopf manifolds}) have pluricanonical metrics. These manifolds
  actually satisfy a much stronger hypothesis: they admit metrics with
  parallel Lee form, i.~e.~$\nabla\theta=0$. Vaisman geometry has close
  links with Sasakian geometry and has been studied intensively in
  recent years. We refer to~\cite{DO, ornea} for the details and
  references on this subject.

\subsection{Weyl harmonic maps}\label{ss:Weyl}

Let $(M,c)$ be a conformal manifold equipped with a Weyl connection
and $(N,h)$ be an arbitrary Riemannian manifold. 

For a fixed metric $g\in c$ consider the equation $tr_g
(\widetilde\nabla Df)=0$ for maps $f\colon M\longrightarrow N$. Here
the differential $Df$ is thought of as a section of
$\Hom(TM,f^*TN)=T^*M\otimes f^*TN$, and $\widetilde\nabla$ is defined
as the tensor product of the dual Weyl connection with the pullback of
the Levi-Civita connection of $(N,h)$. The property of a map to be a
solution of this equation does not depend on the choice of a reference 
metric $g\in c$.
\begin{defn}(\cite{Kok})
A map $f\colon M\longrightarrow N$ from a Weyl manifold $M$ to a
Riemannian manifold $N$ is called Weyl harmonic if it solves 
the equation $tr_g (\widetilde\nabla Df)=0$.
\end{defn}
In the special case when the Weyl connection is the Levi-Civita
connection of $(M,g)$, the defining equation  reduces to the harmonic
map equation. In general, it differs from the usual harmonic map
equation by the term $((n-2)/2)Df(\theta^{\sharp})$, where $n$ is the
real dimension of $M$. Thus, for two-dimensional domains the Weyl
harmonic maps are precisely the usual harmonic maps. 

We refer to~\cite{Kok} for basic existence and uniqueness results for
Weyl harmonic maps. An important  ingredient for our arguments here
concerns pluriharmonicity of Weyl harmonic maps. This occurs when the
K\"ahler--Weyl domain $M$ has complex dimension $2$ or admits a
pluricanonical metric; see~\cite[Theorem~4.2]{Kok}. This motivates the 
following:
\begin{defn}
A K\"ahler--Weyl manifold is called pluri-K\"ahler--Weyl if it is two-dimensional
or if it admits a pluricanonical metric in the sense of Definition~\ref{d:pluri}.
\end{defn}

Another important ingredient in our proofs will be the following strengthened 
version of results of Carlson--Toledo~\cite{CT} and Jost--Yau~\cite{JY}:
\begin{thm}{\rm (\cite{Kok})}\label{t:fact}
Let $M$ be a closed complex manifold and $N$ a Riemannian manifold of
constant negative curvature. If $f\colon M\longrightarrow N$ is a
pluriharmonic map whose rank is at most two and equals two on an open
and dense subset of $M$, then there exists a compact Riemann surface
$S$, a holomorphic map $h\colon M\longrightarrow S$ and a harmonic map
$\phi\colon S\longrightarrow N$ such that $f=\phi\circ h$.
\end{thm}

%\newpage

\section{Fibrations over curves}\label{s:BS}

We now prove the following generalization of Theorem~\ref{t:BSorig}:
\begin{thm}\label{t:BS}
Let $M$ be a closed complex manifold admitting a pluri-K\"ahler--Weyl
structure. Then the following three statements are equivalent:
\begin{enumerate}
\item[I.] $M$ admits a surjective holomorphic map with connected fibers to a closed Riemann surface of genus $\geq 2$,
\item[II.] the fundamental group $\pi_1(M)$ admits a surjective homomorphism to the fundamental group of a closed 
Riemann surface of genus $\geq 2$, and
\item[III.] the fundamental group $\pi_1(M)$ admits a surjective homomorphism to a non-abelian free group.
\end{enumerate}
\end{thm}
\begin{proof}
As noted in the Introduction, each of the statements implies the next one. Thus, we only have to prove that III implies I. 
Clearly we may assume that the complex dimension of $M$ is at least
$2$. Fix a pluri-K\"ahler--Weyl structure on $M$.

Suppose $\pi_1(M)$ surjects onto $F_k$, the free group of rank $k\geq
2$. We can compose this surjection with the homomorphism
$F_k\longrightarrow \pi_1(\Sigma_k)$ given by sending the $i$th
standard generator of $F_k$ to a standard generator of the fundamental
group of the $i$th summand in a decomposition of $\Sigma_k$ as a
connected sum of $k$ tori. This gives us a homomorphism
$\varphi\colon\pi_1(M)\longrightarrow \pi_1(\Sigma_k)$, whose image is
a non-abelian free group; in particular it is not cyclic. Let $f\colon
M\longrightarrow\Sigma_k$ be a smooth map with $f_*=\varphi$ on
$\pi_1(M)$.

Now choose a hyperbolic metric on $\Sigma_k$. Since the image of
$\varphi$ is not cyclic, by~\cite[Theorem~2.2]{Kok} the map $f$ is
homotopic to a Weyl harmonic map, which we also denote by $f$. By
unique continuation, the latter map can not be constant on an
open set. Further, it has rank two on an open and dense subset of $M$,
because otherwise by~\cite[Proposition~1.2]{Kok} its image would
be a closed geodesic, contradicting the fact that the image of $f_*$
is not cyclic.

Recall that the hyperbolic metric on $\Sigma_k$ has non-positive
Hermitian sectional curvature in the sense of Sampson,
cf.~\cite[Chapter 6]{ABCKT}. Therefore, by~\cite[Theorem~4.2]{Kok} any
Weyl harmonic map is pluriharmonic. To summarize, we have a
pluriharmonic map $f\colon M\longrightarrow\Sigma_k$ inducing
$\varphi\colon\pi_1(M)\longrightarrow \pi_1(\Sigma_k)$, and the rank
of the differential $Df$ is two on an open and dense subset of
$M$. Thus we can apply the factorization theorem,
Theorem~\ref{t:fact}, to conclude that $f$ factors through a
holomorphic map $h\colon M\longrightarrow S$ to a compact Riemann
surface. Clearly, the genus of $S$ is at least two, $g(S)\geq 2$.

Consider the Stein factorization of the map $h$
$$
M\stackrel{h_1}{\longrightarrow} C\stackrel{h_2}{\longrightarrow} S,
$$
where $h_1$ has connected fibers. Then $g(C)\geq g(S)\geq 2$, and
$h_1$ has to be non-trivial on $\pi_1(M)$. Since the map $h_1$ is
holomorphic and $M$ is compact, we conclude that $h_1$ is
surjective. Thus the map $h_1$ satisfies all the requirements in 
statement~I., and the theorem is proved.
\end{proof}

\begin{rem}\label{r:max}
Suppose $k$ is the maximal integer for which $\pi_1(M)$ surjects onto
$F_k$. (A maximal $k$ exists and is bounded above by $b_1(M)$.) Then
the genus of $C$ can not be larger than $k$, because $(h_1)_*$ is
surjective on $\pi_1(M)$, and $\pi_1(\Sigma_g)$ surjects onto
$F_g$. The above proof then shows that the genus of $C$ is in fact
equal to $k$.
\end{rem}

\begin{cor}\label{c:free}
A non-abelian free group can not be the fundamental group of a closed
pluri-K\"ahler--Weyl manifold.
\end{cor}
\begin{proof}
Suppose that $M$ is a pluri-K\"ahler--Weyl manifold whose fundamental
group $\pi_1(M)$ is isomorphic to $F_k$, where $k\geq 2$. Fix an
isomorphism $\varphi\colon\pi_1(M)\longrightarrow F_k$. The proof of
Theorem~\ref{t:BS} shows that $\varphi$ factors through a surjection
to the fundamental group of a closed Riemann surface, and thus can not
be injective -- a contradiction.
\end{proof}

Replacing the hyperbolic surfaces in Theorem~\ref{t:BS} by closed
hyperbolic manifolds of higher dimension, we obtain the following
generalization:
\begin{thm}\label{t:main}
Let $M$ be a closed complex manifold admitting a pluri-K\"ahler--Weyl
structure, and $N$ a closed Riemannian manifold of constant negative
curvature. If $\varphi\colon\pi_1(M)\longrightarrow\pi_1(N)$ is a
representation with non-cyclic image, then there exists a compact
Riemann surface $S$ and a holomorphic map $h\colon M\longrightarrow S$ 
with connected fibers such that $\varphi$ factors through $h_*$. 
\end{thm}
Note that under these hypotheses the statements I, II and III in
Theorem~\ref{t:BS} hold for $M$. The proof of Theorem~\ref{t:main} is
similar to the one of Theorem~\ref{t:BS}; cf. also the proof
of~\cite[Theorem~5.5]{Kok}, where the situation when $\varphi$ is an
isomorphism is considered. In slightly more detail, there is a Weyl
harmonic map from $M$ equipped with a pluri-K\"ahler--Weyl structure
to $N$ with its hyperbolic metric which induces $\varphi$ on
$\pi_1(M)$. This map is pluriharmonic and has rank at least two on an
open and dense subset of $M$, just as in the above proof of
Theorem~\ref{t:BS}. By~\cite[Corollary~4.6]{Kok}, the rank is at most
two. Therefore, the conclusion follows using Theorem~\ref{t:fact}.

As an immediate consequence of Theorem~\ref{t:main} we have the following
result, which is interesting in the context of the relation $M\geq N$ on manifolds defined 
by the existence of maps $M\longrightarrow N$ of non-zero degree. We refer to~\cite{CT,KL}
for further information on this relation.
\begin{cor}\label{c:dom}
Let $M$ be a  closed pluri-K\"ahler--Weyl manifold, and $N$ a closed real hyperbolic
manifold of dimension $\geq 4$. Then every map $f\colon M\longrightarrow N$ has degree zero.
\end{cor}
\begin{proof}
If the degree were non-zero, then $f_*\colon\pi_1(M)\longrightarrow\pi_1(N)$ 
would be surjective onto a finite index subgroup. By Theorem~\ref{t:main}
the harmonic map in the homotopy class of $f$ would factor through a two-dimensional
manifold, showing that the degree had to be zero after all.
\end{proof}

In~\cite[Theorem~5.5]{Kok} it was proved that the fundamental groups
of closed hyperbolic manifolds of dimension $\geq 3$ can not be
fundamental groups of closed pluri-K\"ahler--Weyl manifolds. We now
extend this result:
\begin{cor}\label{c:central}
Let $N$ be a closed hyperbolic manifold of dimension $\geq 3$. Any
group $\Gamma$ which fits into a central extension of the form
$$
1\longrightarrow\Z^k\longrightarrow\Gamma\stackrel{\varphi}{\longrightarrow}
\pi_1(N)\longrightarrow 1
$$
can not be the fundamental group of a closed pluri-K\"ahler--Weyl
manifold.
\end{cor}
\begin{proof}
Let $M$ be a K\"ahler--Weyl manifold whose fundamental group is
$\Gamma$. Then, by Theorem~\ref{t:main}, $\varphi$ factors through a
surjective homomorphism $f_*\colon\Gamma\longrightarrow\pi_1(S)$, for
some closed Riemann surface $S$ of genus $\geq 2$. Since $\pi_1(S)$
has trivial center, the surjection $f_*$ descends from $\Gamma$ to
the quotient $\pi_1(N)$. However, the identity of $\pi_1(N)$ can not
factor through $\pi_1(S)$, and we obtain a contradiction.
\end{proof}
Note that the dimension assumption on $N$ can not be dropped, as the
direct product of $\Z^2$ with any surface group is the fundamental
group of a compact complex, in fact K\"ahler, surface.
\begin{ex}
Let $N$ be a closed hyperbolic $3$-manifold, and $M\longrightarrow N$
a circle bundle. Then the total space $M$ does not admit a complex
structure by Corollary~\ref{c:central}. A weaker statement in this
direction is proved in~\cite{D} using the Enriques--Kodaira
classification.
\end{ex}

\section{Twistor spaces}\label{s:twistor}

An oriented Riemannian four-manifold is called
half-conformally flat if its Weyl tensor is either self-dual or
anti-self-dual. This is a conformally invariant condition. By work of
Penrose and Atiyah--Hitchin--Singer~\cite{AHS}, a half-conformally
flat four-manifold $N$ has associated to it a complex three-fold $Z$,
called its twistor space, which is differentiably a two-sphere bundle
bundle over $N$. In particular it has the same fundamental group as
$N$. It is a theorem of Taubes~\cite{Taubes} that every closed
oriented four-manifold admits a metric with anti-self-dual Weyl tensor
after stabilization by connected summing with many copies of
$\blowup$. In particular every finitely presentable group is the
fundamental group of a compact complex three-fold obtained as the
twistor space of a suitable four-manifold.

Recall that a discrete group is called large if it has a finite
index subgroup that admits a surjective homomorphism to $F_2$. This
notion was introduced by Gromov~\cite{gromov}, and has many important
ramifications, for example in geometric group theory and in spectral
geometry. For twistor spaces we have:
\begin{thm}\label{t:twistor}
Let $N$ be a closed half-conformally flat four-manifold with large
fundamental group. Then its twistor space is a complex manifold that
does not admit any pluri-K\"ahler--Weyl structure.
\end{thm}
This should be compared with a result of Hitchin~\cite{Hitchin}, who
showed that the only K\"ahler twistor spaces are those of $S^4$ and of
$\blowup$; in particular K\"ahler twistor spaces are simply
connected. Theorem~\ref{t:twistor} shows that under the weaker
pluri-K\"ahler--Weyl assumption we can still conclude that the
fundamental group is not large.
\begin{proof}
The assumption about the fundamental group means that after replacing
$N$ by some finite covering, its fundamental group surjects to
$F_2$. As finite covers of pluri-K\"ahler--Weyl manifolds are
pluri-K\"ahler--Weyl, and the twistor space has the same fundamental
group as the four-manifold, we may assume for a contradiction that we
have a twistor space $Z$ whose fundamental group surjects to $F_2$. By
Theorem~\ref{t:BS} this implies that $Z$ fibers holomorphically with
connected fibers over a curve $C$ of genus $\geq 2$. The pullback by
$h\colon Z\longrightarrow C$ is injective on $H^1(C;\C)$. This
cohomology group has a Hodge decomposition, and so we obtain a
contradiction as soon as we see that $Z$ has no holomorphic
one-forms. But this is a standard fact, compare~\cite[p.~27]{ABCKT} or 
Lemma~\ref{l:vanish} below.
%Every fiber $F$ of the twistor fibration $Z\longrightarrow X$ is a
%complex projective line with normal bundle
%$\OO(1)\oplus\OO(1)$. Therefore we have the following exact sequence
%of holomorphic vector bundles over $F$:
%$$
%0\longrightarrow \OO_F(-1)\oplus \OO_F(-1)\longrightarrow\Omega^1_Z\vert_F\longrightarrow\Omega^1_F\longrightarrow 0 \ .
%$$
%The bundles on the left and on the right have no non-zero holomorphic sections, and so every holomorphic one-form
%on $Z$ vanishes identically on $F$. But the fibers $F$ cover $Z$, which therefore has no holomorphic one-forms.
\end{proof}
\begin{rem}
Sometimes $N$ can be chosen as a complex K\"ahler surface with
anti-self-dual Weyl tensor, e.g. $\C P^1\times C$, for a curve
$C$ of genus $\geq 2$. In these cases the twistor space $Z$ is
diffeomorphic to $\bP (\OO_N\oplus\OO_N(K_N))$, where $K_N$ is the
canonical bundle of $N$. This shows that the twistor complex structure
is not pluri-K\"ahler--Weyl although the smooth manifold underlying
$Z$ also carries a K\"ahler complex structure.
\end{rem}

\begin{ex}\label{ex:ball}
There are complex algebraic surfaces $N$ which are ball quotients $\C
H^2/\Gamma$, and which fiber holomorphically over curves of genus
$\geq 2$, cf.~for example \cite{BHH}. In particular, the latter
property implies that they have large fundamental groups. Since the
Bergmann metric on $\C H^2$ has self-dual Weyl tensor, the twistor
space $Z$ of such an $N$ is a complex manifold. By
Theorem~\ref{t:twistor}, it can not admit a pluri-K\"ahler-Weyl
structure.
\end{ex}
We shall show in Corollary~\ref{c:twist} below that the largeness of
the fundamental group can be dispensed with in this example.

There is a well known generalization of the construction of twistor
spaces for oriented conformally flat manifolds of arbitrary even
dimension $2n$. We briefly recall this construction, referring
to~\cite{Bianchi,Slup,BMN} for further details.

Let $(N,h)$ be an oriented Riemannian $2n$-manifold, and $Z$ the
quotient of its oriented orthonormal frame bundle by $U(n)\subset
SO(2n)$. This is the bundle of pointwise orthogonal complex structures
on $N$ compatible with the orientation. The total space $Z$, called
the twistor space of $N$, carries a tautological almost complex
structure, which is integrable if $h$ is conformally flat. In this
case $Z$ is a complex manifold, and the fibers of the projection
$Z\longrightarrow N$ are holomorphic submanifolds isomorphic to the
Hermitian symmetric space $X_n = SO(2n)/U(n)$. For example, if
$N=S^{2n}$ is a round sphere, the twistor space is $X_{n+1}$, with the
fiber $X_n$ embedded in the standard way.

\begin{lem}\label{l:vanish}
Let $Z$ be the twistor space of a conformally flat manifold $N$. Then
$Z$ has no non-trivial holomorphic one-forms.
\end{lem}
\begin{proof}
The fibers of the projection $\pi\colon Z\longrightarrow N$ sweep out
$Z$, so it is enough to show that for a fiber $F$ the restriction
$\Omega^1_Z\vert_F$ has no non-trivial holomorphic sections. Taking a
conformal chart for $N$ around $\pi(F)$, we can identify
$\Omega^1_Z\vert_F$ with $\Omega^1_{X_{n+1}}\vert_{X_n}$. The natural
K\"ahler--Einstein metric of $X_{n+1}$ of positive Ricci curvature
induces a metric of negative mean curvature on this bundle, which
therefore has no holomorphic sections by a standard application of the
Bochner vanishing argument, cf.~\cite[Chapter III]{Kob}.
\end{proof}

Combining Theorem~\ref{t:BS} and Lemma~\ref{l:vanish}, we conclude:
\begin{thm}\label{t:cf}
The twistor space of a closed conformally flat manifold with large
fundamental group is a complex manifold which is not
pluri-K\"ahler--Weyl.
\end{thm}
This generalizes the results of~\cite{Slup,Cam,BMN}, showing that
K\"ahlerian twistor spaces are simply connected. There are many
conformally flat manifolds to which Theorem~\ref{t:cf} applies. For
example, it is well known that connected sums of conformally flat
manifolds are again conformally flat. Thus, given any two conformally
flat manifolds of the same dimension with positive first Betti
numbers, their connected sum satisfies the hypotheses of the theorem.

For some conformally flat manifolds more can be said than in
Theorem~\ref{t:cf}: not only is the twistor space not
pluri-K\"ahler--Weyl, but in fact no manifold with the same
fundamental group is pluri-K\"ahler--Weyl. This is so for the
conformally flat manifolds $(S^1\times S^{2n-1})\#\ldots\# (S^1\times
S^{2n-1})$ by Corollary~\ref{c:free}, and for real hyperbolic
manifolds by ~\cite[Theorem~5.5]{Kok}. Here are some more examples:
\begin{ex}
Let $N$ be any closed oriented real hyperbolic manifold of dimension
$2n-1\geq 3$. Then $N\times S^1$ with the product metric is
conformally flat. Its twistor space $Z$ has fundamental group
$\pi_1(N)\times\Z$, and no manifold with such a fundamental group can
be pluri-K\"ahler--Weyl by Corollary~\ref{c:central}. Some hyperbolic manifolds 
also have non-trivial circle bundles over them whose total spaces are conformally flat,
see~\cite{Bel}, and Corollary~\ref{c:central} applies to their
fundamental groups as well.
\end{ex}

\begin{rem}
There is a notion of twistor spaces for quaternionic K\"ahler manifolds in the sense of Salamon~\cite{Sal}. Our discussion
could be generalized to this case, but the generalization would be vacuous. Quaternionic K\"ahler manifolds are 
always Einstein, and so have constant scalar curvature. If the scalar curvature %of a quaternionic K\"ahler manifold $X$ 
is positive, then Salamon~\cite{Sal} proved that the twistor space is K\"ahler and simply connected. If the scalar curvature 
is negative, Semmelmann and Weingart~\cite{SW} proved that the first Betti number vanishes, and so every locally conformally 
K\"ahler structure on the twistor space would in fact be K\"ahler. In the K\"ahler case, Campana~\cite{Cam} showed that the 
fundamental group is
trivial. Finally, in the case of zero scalar curvature, the Cheeger--Gromoll splitting theorem shows that the fundamental group
is virtually Abelian, and therefore not large.
\end{rem}

\section{Further applications}\label{s:appl}

\subsection{When K\"ahler--Weyl implies K\"ahler}

Vaisman has put forward the philosophy that K\"ahler--Weyl manifolds
which are, in a suitable sense, topologically K\"ahler, should in fact
be K\"ahler, cf.~\cite{V,DO}. In this direction, he proved the following:
\begin{prop}{\rm (\cite{V})}\label{p:Vais}
A compact locally conformally K\"ahler manifold which admits some K\"ahler
metric, or, more generally, which satisfies the $\del\bar\del$-Lemma, is globally conformally K\"ahler.
\end{prop}
We shall give a proof, following~\cite{DO}, because we need the argument for our next result.
\begin{proof}
Suppose $(M,g,J,\omega)$ is locally conformally K\"ahler of complex dimension $n$, 
with Lee form $\theta$. Let $\alpha = \theta \circ J$ be the anti-Lee form. The definition
of the Lee form implies
$$
\alpha = -\frac{1}{n-1}d^*_g\omega \ ,
$$
where $d^*_g$ is the formal $L^2$-adjoint of the exterior derivative with respect to $g$; compare~\cite{DO}. If we can
globally conformally rescale the metric so that for the new metric the corresponding $\alpha$ is closed, then $\alpha$
is closed and coclosed and therefore harmonic. However, it is also in the image of $d^*$, and so the 
Hodge decompostition theorem implies that $\alpha$ vanishes. This means that the Lee form vanishes,
and so the new metric is K\"ahler.

We know that $\theta$ is always a closed real one-form. Decomposing $d\theta$ into $(p,q)$-types, we find that 
$\del\theta^{1,0}=0=\bar\del\theta^{0,1}$ and $\del\theta^{0,1}+\bar\del\theta^{1,0}=0$.

The anti-Lee form is $\alpha = i (\theta^{1,0}-\theta^{0,1})$. Now the above relations obtained from $d\theta=0$
imply that $d\alpha=2i\bar\del\theta^{1,0}$. Thus $d\alpha$ is an exact form of pure type $(1,1)$. 
Therefore the $\del\bar\del$-Lemma implies that there is a globally defined real function $\varphi$ on $M$
such that $d\alpha=2i\del\bar\del\varphi$.

Consider the metric $h = \exp (\varphi)g$. This is locally conformally K\"ahler with fundamental two-form 
$\exp (\varphi)\omega$ and Lee form $\theta+d\varphi$. Its anti-Lee form is $\alpha+i(\del\varphi-\bar\del\varphi)$,
which is closed because $d\alpha=2i\del\bar\del\varphi$. This completes the proof.
\end{proof}

We now extend Vaisman's result in the following way:
\begin{prop}\label{p:fiber}
Let $M$ be a closed locally conformally K\"ahler manifold admitting a holomorphic
map $f\colon M\longrightarrow N$ to a complex manifold $N$ which is 
K\"ahler, or at least satisfies the $\del\bar\del$-Lemma. Assume that 
$f^*\colon H^1(N;\bR)\longrightarrow H^1(M;\bR)$ is an isomorphism. 
Then $M$ is globally conformally K\"ahler.
\end{prop}
\begin{proof}
Since $f^*$ is an isomorphism, the cohomology class of the Lee form $\theta$ is the pullback
of some class on $N$. After rescaling the given locally conformally K\"ahler metric on $M$,
we may assume that the Lee form $\theta$ is itself the pullback of a closed form $\beta$ on $N$. 
As $f$ is holomorphic and $f^*\beta = \theta$, we find that the anti-Lee form $\alpha = \theta\circ J$
equals $f^*(\beta\circ J)$, where we also use $J$ to denote the complex structure on $N$, not just 
on $M$. As in the previous proof, the $\del\bar\del$-Lemma, applied now on $N$, not on $M$,
tells us that there is a real-valued function $\varphi$ on $N$ such that $d(\beta\circ J)=2i\del\bar\del\varphi$.
Pulling back to $M$ we find $d\alpha = 2i\del\bar\del (\varphi\circ f)$. As in the previous proof, the 
metric $\exp (\varphi\circ f)g$ on $M$ is K\"ahler, because its Lee form vanishes identically.
\end{proof}
This proposition allows us to prove the following:
\begin{thm}
Let $M$ be a closed pluri-K\"ahler--Weyl manifold whose fundamental
group admits a surjection $\varphi\colon\pi_1(M)\longrightarrow
\pi_1(\Sigma_g)$ with $g\geq 2$ for which $\varphi^*\colon
H^1(\Sigma_g;\bR)\longrightarrow H^1(M;\bR)$ is an isomorphism. Then
$M$ admits a K\"ahler metric.
\end{thm}
\begin{proof}
First suppose that $M$ is a complex surface. Since $f^*$ is
an isomorphism on the first cohomology, $M$ has even first Betti
number because $N$ does. A compact complex surface with even first Betti number 
is K\"ahler, see Buchdahl~\cite{B} for a proof that is independent of 
the Kodaira classification.

If the complex dimension of $M$ is at least $3$, then we apply Theorem~\ref{t:BS} and Remark~\ref{r:max} to 
obtain a holomorphic map to a closed complex curve inducing an isomorphism in the first cohomology. In this 
case $M$ is locally conformally K\"ahler, and applying Proposition~\ref{p:fiber} to the fibration over a curve shows
that it is in fact globally conformally K\"ahler.
\end{proof}

We can also apply Proposition~\ref{p:fiber} to the twistor spaces of
ball quotients to obtain the following generalization of
Example~\ref{ex:ball}:
\begin{cor}\label{c:twist}
Let $N = \C H^2/\Gamma$ be a compact ball quotient. Then its twistor
space $Z$ is a complex manifold which is not pluri-K\"ahler--Weyl.
\end{cor}
\begin{proof}
Consider the projection $\pi\colon Z\longrightarrow N$. Its fibers are
two-spheres, and so the Euler class of the vertical tangent bundle
does not vanish when evaluated on a fiber. Thus the fibers are not
null-homologous, and any smooth map homotopic to $\pi$ must have
maximal rank somewhere.

Now suppose $Z$ is pluri-K\"ahler--Weyl. Then $\pi$ is homotopic to a
Weyl harmonic map by the results of~\cite{Kok}. As the rank of the
Weyl harmonic representative in the homotopy class of $\pi$ is greater
than $2$, it is holomorphic by~\cite[Corollary~4.6]{Kok}. Now we can
apply Proposition~\ref{p:fiber} to this holomorphic map to conclude
that $Z$ is K\"ahler. But this contradicts the results
of~\cite{Hitchin,Cam}.
\end{proof}

\subsection{Period domains}

Denote by $X$ the symmetric space $SO(2p,q)/(SO(2p)\times SO(q))$;
it is Hermitian symmetric if and only if $p=1$ or $q=2$. By the general 
construction in~\cite{GS}, the space $X$ has associated to it a 
homogeneous complex manifold $D=SO(2p,q)/(U(p)\times SO(q))$, 
which one can think of as a twistor space; here the isotropy group of $D$ 
is the centralizer of a torus in $SO(2p)\times SO(q)$. 
The manifold $D$ parametrises Hodge structures of weight $2$ and is 
an example of a Griffiths period domain; see~\cite{GS,ABCKT}
and references there for the details.

The following theorem generalises an analogous result of
Carlson and Toledo in the K\"ahler setting, cf.~\cite[Chapter~6]{ABCKT}.
\begin{thm}\label{t:period} 
Let $\Gamma$ be a torsion-free cocompact lattice in $SO(2p,q)$, where
$p>1$ and $q>2$. Then the compact complex manifold $D/\Gamma$ is not
homotopy equivalent to a compact pluri-K\"ahler--Weyl manifold.
\end{thm}
\begin{proof}
Let $\pi\colon D\longrightarrow X$ be the natural topologically
trivial fibration with fiber $X_p=SO(2p)/U(p)$. We denote by $c_1$ 
the Chern class of the anti-canonical bundle of $D/\Gamma$ and let
$\Omega$ be a volume form on $X/\Gamma$. Since the fiber $X_p$ is a
compact Hermitian symmetric space of positive Ricci curvature, its anti-canonical bundle is positive;
thus $c_1^r>0$ on $X_p$, where $r$ is the complex dimension of $X_p$. By the
Fubini theorem
$$
\int_{D/\Gamma}\pi^*\Omega\wedge c_1^r>0
$$
and, therefore, the cohomology class of $\pi^*\Omega$ is non-zero.
Suppose there exists a compact pluri-K\"ahler--Weyl manifold $M$ and a
homotopy equivalence $f\colon M\longrightarrow D/\Gamma$. Then the map
$g=\pi\circ f\colon M\longrightarrow X/\Gamma$ is non-trivial in top-degree
cohomology, contradicting~\cite[Corollary~5.2]{Kok}. 
\end{proof}
\begin{rem}
It is not hard to see that the natural complex structure on $D/\Gamma$
is not K\"ahler. This follows from the known fact that the deformation
space of a fiber of $\pi$ is not compact. However, the generalization
to the class of pluri-K\"ahler--Weyl manifolds is already non-trivial.
The point of the theorem is, of course, that no complex structure on a
homotopy equivalent manifold can be pluri-K\"ahler--Weyl.
\end{rem}
\begin{rem}
Theorem~\ref{t:period} applies to all the locally homogeneous complex
manifolds studied by Griffiths and Schmidt~\cite{GS}, as long as $X$ is 
not Hermitian symmetric and $X/\Gamma$ has non-zero Euler number. The
condition on the Euler number is required for the applicability of the results
of~\cite{Kok}.
\end{rem}

\section{Postscript}

The original motivation for this work was the extension of topological
restrictions on K\"ahler manifolds to all K\"ahler--Weyl manifolds. It is clear 
that some such extensions are possible, as shown by the fact that, in complex dimension
$\geq 3$, K\"ahler--Weyl manifolds are locally conformally K\"ahler and,
if the first Betti number vanishes, globally conformally K\"ahler and therefore
K\"ahlerian. This leads to:
\begin{ex}\label{ex:pi1}
Let $\Gamma$ be a finitely presentable group with $b_1(\Gamma)=0$
that is not a K\"ahler group. Then $\Gamma$ can not be the fundamental 
group of any K\"ahler--Weyl manifold.
\end{ex}
We do not have to worry about complex dimension two in this example, because in
that dimension the vanishing of the first Betti number implies that a complex manifold
is K\"ahler.

Example~\ref{ex:pi1}
applies in many concrete instances, e.~g.~when $\Gamma$ is the fundamental
group of a hyperbolic homology sphere, or when it is a free product of two non-trivial 
groups with vanishing first Betti number; compare~\cite{ABCKT}. Thus these groups are not fundamental
groups of K\"ahler--Weyl or locally conformally K\"ahler manifolds in any dimension.
Recall however that, by the result of Taubes~\cite{Taubes}, every finitely presentable group
is the fundamental group of a compact complex three-fold. 

There are also restrictions that apply to specific complex structures only, rather than to all
complex structures on a given manifold.
\begin{ex}\label{ex:prod}
Let $M$ be a compact complex manifold that admits a holomorphic map $f\colon M\longrightarrow B$
to a K\"ahler manifold $B$ for which $f^*$ is an isomorphism on degree one cohomology with real coefficients, 
and some smooth fiber $F=f^{-1}(p)$ is not K\"ahlerian. Then $M$ can not be locally conformally K\"ahler. 
\end{ex}
The assumption about $f^*$ implies, via Proposition~\ref{p:fiber}, that if $M$ is locally conformally 
K\"ahler, then it is K\"ahlerian. In this case, the complex submanifold $F\subset M$ also has to be 
K\"ahlerian, leading to a contradiction.

For technical reasons, which stem from the work in~\cite{Kok}, in this paper we have only proved the 
Siu--Beauville theorem for pluri-K\"ahler--Weyl rather than for all K\"ahler--Weyl manifolds. In fact
the stronger assumption is only needed for one step in the proof, to conclude that Weyl harmonic
maps are pluriharmonic.
It remains an open problem to decide whether the result is true for all K\"ahler--Weyl manifolds. 

After this paper
was written, Ornea and Verbitsky~\cite{OV} pointed out that pluri-K\"ahler--Weyl manifolds are topologically
Vaisman, and that this point of view gives alternative proofs of some of the corollaries of Theorem~\ref{t:BS}.

In summary, there are many compact complex manifolds that are not K\"ahler--Weyl, and there are also
K\"ahler--Weyl manifolds that are not pluri-K\"ahler--Weyl.

%\bigskip

\bibliographystyle{amsplain}

\bigskip

\end{document}